\documentclass{amsart}

\usepackage{hyperref}
\usepackage[T1]{fontenc}
\usepackage{times}
\usepackage{amssymb,verbatim}
\theoremstyle{plain}
\usepackage[T1]{fontenc}
\usepackage[utf8]{inputenc}
\usepackage[english]{babel}
\usepackage{amssymb}
\usepackage{amsthm}
\usepackage{graphics}
\usepackage{amsmath}
\usepackage{amstext}
\usepackage{multirow}
\usepackage{graphicx}
\usepackage{amsmath}
\usepackage{amssymb}
\usepackage{amsthm}
\usepackage{amstext}
\usepackage{mathrsfs}
\usepackage{amscd}
\usepackage{mathtools}
\usepackage{tikz}
\usetikzlibrary{cd}
\usetikzlibrary{matrix}

\usepackage{todonotes}

\newtheorem{thm}{Theorem}[section]
\newtheorem{lemma}[thm]{Lemma}
\newtheorem{prop}[thm]{Proposition}

\newtheorem{remark}[thm]{Remark}

\newcommand{\mb}{\mathbb}
\newcommand{\mc}{\mathcal}

\newcommand{\C}{\mb C}

\newcommand{\N}{\mb N}

\newcommand{\M}{\mathcal M}

\newcommand{\F}{\mc F}
\newcommand{\G}{\mc G}

\newcommand{\VF}{\mathfrak X}
\newcommand{\VFR}{\widehat{\VF}}

\DeclareMathOperator{\id}{id}

\DeclareMathOperator{\ad}{ad}
\DeclareMathOperator{\Tr}{Tr}
\DeclareMathOperator{\SL}{SL}

\DeclareMathOperator{\Der}{Der}
\DeclareMathOperator{\End}{End}

\numberwithin{equation}{section}

\usepackage{mathtools}
\usepackage{ifmtarg}

\addtocounter{section}{0}             
\numberwithin{equation}{section}       

\title[Primitive Lie algebras of rational vector fields]{Primitive Lie algebras of rational vector fields}

\author[G. Casale, F. Loray, J.V. Pereira and F. Touzet ]
{Guy CASALE$^1$, Frank LORAY$^1$,  Jorge Vit\'orio PEREIRA$^{2}$ and Fr\'ed\'eric TOUZET$^1$}
\address{\newline $1$  Univ Rennes, CNRS, IRMAR - UMR 6625, F-35000 Rennes, France\hfill\break
$2$ IMPA, Estrada Dona Castorina, 110, Horto, Rio de Janeiro,
Brasil} \email{$^1$ guy.casale@univ-rennes1.fr, frank.loray@univ-rennes1.fr, \newline frederic.touzet@univ-rennes1.fr}
\email{$^2$ jvp@impa.br}
\thanks{We thank CNRS, Universit\'e de Rennes 1, ANR-11-LABX-0020-0 program Henri Lebesgue Center,  ANR-16-CE40-0008 project ``{\it Foliage}'' and CAPES-COFECUB Ma 932/19 project. The third author was
supported by CNPq and FAPERJ}
\date{\today}
\makeatletter

\AtBeginDocument{%
  \@ifpackageloaded{refcheck}
  {%
    \@ifundefined{hyperref}{}{%
      \let\T@ref@orig\T@ref%
      \def\T@ref#1{\T@ref@orig{#1}\wrtusdrf{#1}}%
      \let\@refstar@orig\@refstar%
      \def\@refstar#1{\@refstar@orig{#1}\wrtusdrf{#1}}
      \DeclareRobustCommand\ref{\@ifstar\@refstar\T@ref}%
    }%
  }{}%
}

\makeatother

\begin{document}

\begin{abstract}
    Let $\mathfrak g$ be a transitive, finite-dimensional Lie algebra  of rational vector fields on a projective manifold.
    If $\mathfrak g$ is primitive, i.e., does not locally preserve any  foliation, then it determines a rational map to
    an algebraic homogenous space $G/H$ which maps $\mathfrak g$ to $\mathrm{Lie}(G)$.
\end{abstract}

\maketitle

\setcounter{tocdepth}{1}


\section{Introduction}
Integration of transitive Lie algebras $\mathfrak g$ of complete vector fields on a differentiable manifold $X$ gives rise to transitive actions of a Lie group $G$ on $X$. Through this action, $X$ is locally identified with $G/H$ where $H$ is a Lie group integrating the isotropy subalgebra $\mathfrak h \subset \mathfrak g$. In general, the action can be highly transcendental even if the manifold and vector fields in question are algebraic.

This paper studies transitive Lie algebras of rational vector fields on complex projective manifolds.

By a rational vector field on $X$, we mean a derivation of the field of rational functions on $X$.
Such a derivation induces a holomorphic vector field on a Zariski open subset of $X$, which extends meromorphically to the whole of $X$.
Our main results, roughly speaking, give sufficient conditions for the existence of an algebraic (hence global) conjugation, or more generally semi-conjugation, of the action with the standard model $G/H$.

The case where $X$ has dimension $1$ was considered in \cite{MR3842065}: such a Lie algebra
is naturally deduced from the Frobenius integrability of codimension $1$ foliations on transversal curves.
A factorization through $G/H$ is therefore established  and used in \cite[Theorem 6.5]{MR3842065}
(see \ref{SS:curves}) to prove some stronger integrability property for the foliation. We expect the results of this paper to be useful for investigating the transverse dynamics of higher codimension foliations on projective manifolds.

\subsection{Algebraic normalization}
Let $X$ be a projective manifold and let $\mathfrak g$ be a finite-dimensional Lie subalgebra of the Lie algebra of rational vector fields on $X$. Suppose $p \in X$ is any sufficiently general point. In that case, we define the isotropy subalgebra of $\mathfrak g$ at $p$, denoted by $\mathfrak h_p$,  as the subalgebra of $\mathfrak g$ consisting of vector fields vanishing at $p$.

A Lie algebra $\mathfrak g$ as above is called \textit{transitive} if it contains a basis of the $\mathbb C(X)$-vector space
of rational vector fields on $X$. When $\mathfrak g$ is transitive, two sufficiently general points $p,p' \in X$ have isomorphic (through an isomorphism of $\mathfrak g$) isotropy subalgebras $\mathfrak h_p$ and $\mathfrak h_{p'}$. Thus when $\mathfrak g$ is transitive, we can safely refer to the isotropy subalgebra $\mathfrak h$, meaning the isotropy subalgebra for a sufficiently general point $p \in X$. We will denote the normalizer of $\mathfrak h$ in $\mathfrak g$ by $N_{\mathfrak g}(\mathfrak h)$, i.e.,
\[
    N_{\mathfrak g} (\mathfrak h) = \{ v \in \mathfrak g \, \big\vert \,  \ad(v) ( \mathfrak h) \subset \mathfrak h \} \, ,
\]
where $\ad(v)(\cdot) = [ v , \cdot]$ stands for the adjoint action.

We will say that a Lie algebra  is  \textit{complete}, following  \cite[Chapter I, Section 3]{MR559927},  if
all its derivations are inner derivations, and if it has a trivial center.

For a Lie group $G$ and a Lie subgroup $H$, we will denote by $G/H$ the orbit space of the action of $H$ on $G$ by right translations
($(h,g)\mapsto gh^{-1}$). Elements of $G/H$ are also called left cosets. Right invariant vector fields on $G$ are infinitesimal generators of the action of $G$ on itself by left translations. They commute with the infinitesimal generators of the action by right translations,  the left-invariant vector fields.

The result below is an alternative version of \cite[Theorem 5.12]{MR3719487}, which, although less general, is formulated in terms of intrinsic properties of the Lie algebra $\mathfrak g$ and of its isotropy subalgebra $\mathfrak h$. This point is discussed in subsection \ref{S:ingredient}.

\begin{thm}\label{THM:A}
    Let $\mathfrak g$ be a finite-dimensional Lie algebra of rational vector fields on a projective manifold $X$ with isotropy
    algebra equal to $\mathfrak h$.
    If $\mathfrak g$ is transitive, complete and $N_{\mathfrak g}(\mathfrak h) = \mathfrak h$ then there exists a dominant rational map $\varphi : X \dashrightarrow G/H$ to a homogeneous algebraic space such that $\mathfrak g$ coincides with the pull-back under $\varphi$ of the Lie
    algebra of fundamental vector fields on $G/H$.
\end{thm}

As usual, by \textit{fundamental vector fields} on $G/H$, we mean vector fields induced by right invariant vector fields on $G$.

Theorem \ref{THM:A} is not optimal. As shown by Proposition \ref{P:affine}, there exist \textit{ non-complete} transitive Lie algebras of rational vector fields, for instance the Lie algebra $\mathfrak{aff}_n^0=\mathfrak{sl}_n\ltimes \C^n$ of the group of affine motions of $\mathbb C^n$ ($n\geq 2$) with determinant $1$ for which the conclusion of Theorem \ref{THM:A} still holds.  Indeed, one can exhibit an outer derivation on $\mathfrak{aff}_n^0$ by considering the Lie bracket with the radial vector field $R=\sum_{i=1}^n x_i {\partial}/{\partial x_i}$.  It acts trivially on the isotropy subalgebra $\mathfrak{sl}_n$ and non-trivially on $\C^n$.

Nevertheless, we point out that the assumption $N_{\mathfrak g}(\mathfrak h) = \mathfrak h$
cannot be disregarded entirely.  Indeed there are transitive actions of the complete Lie algebra $\mathfrak{sl}(2, \mathbb C)$ on $\mathbb C^3$ which are transitive but not algebraically conjugated to the Lie algebra of right invariant vector fields on $\SL(2,\mathbb C)$, see \cite[Section 3.5]{MR3719487} for examples. The paper \cite{MR2354208} thoroughly studies other explicit examples of Lie algebras of polynomial vector fields on $\mathbb C^3$ isomorphic to $\mathfrak{sl}(2,\mathbb C)$.

\subsection{Primitive Lie algebras}\label{SS:Prim}
A transitive finite-dimensional Lie algebra of rational vector fields $\mathfrak g$ on $X$ is called \textit{primitive}
if its isotropy subalgebra $\mathfrak h$ is maximal among all Lie subalgebras
of $\mathfrak g$. The existence of an intermediate Lie algebra $\mathfrak h\subset\mathfrak l\subset\mathfrak g$ would correspond at a general point $p$ of $X$ to the existence of a foliation
on the germ of neighborhood $(X,p)\simeq(\C^n,0)$ invariant under $\mathfrak g$. Indeed, such Lie algebra $\mathfrak l$ provides a vector subspace
$\mathfrak l (0)$ of the tangent space of $(\C^n,0)$ at the origin isomorphic to $\mathfrak l/\mathfrak h$; one can use transitivity of $\mathfrak g$ to propagate $\mathfrak l(0)$ into a distribution $\mathcal D$ on $(\mathbb C^n,0)$. There is no ambiguity in the definition of $\mathcal D$ since the isotropy algebra $\mathfrak h$ preserves $\mathfrak l$.
The distribution $\mathcal D$ turns out to be Frobenius integrable: the leaf through $0$
is the local orbit of $\mathfrak l$ through $0$, and the other leaves are obtained by $\mathfrak g$-translation. See Section \ref{S:transitive} where we discuss the relationship between intermediate subalgebras containing the isotropy and invariant foliations in a more general setting.

\begin{thm}\label{THM:B}
    Let $\mathfrak g$ be a finite-dimensional Lie algebra of rational vector fields on a projective manifold $X$ of dimension at least two.
    If $\mathfrak g$ is transitive and primitive,  then there exists a rational map $\varphi : X \dashrightarrow G/H$
    to a homogeneous algebraic space such that $\mathfrak g$ coincides with the pull-back under $\varphi$ of the Lie algebra induced by the fundamental vector fields on $G/H$.
\end{thm}

\begin{remark}
    When $X$ has dimension one, the conclusion of Theorem \ref{THM:B} remains true except when  $\mathfrak g$ is  one-dimensional (see  Proposition \ref{P:rationalLie} ). Observe also that, in dimension two or higher, primitiveness automatically implies that $\mathfrak g$ is non-abelian.
\end{remark}

\begin{remark} \label{R: heqnormh}
    The hypothesis of Theorem  \ref{THM:B} force the equality $N_{\mathfrak g}(\mathfrak h) = \mathfrak h$. Indeed, thanks to the inexistence of intermediate Lie subalgebras, either $N_{\mathfrak g}(\mathfrak h) = \mathfrak h$, either $N_{\mathfrak g}(\mathfrak h) = \mathfrak g$. In the latter case, this would imply that $\mathfrak h$ is an ideal and then would provide for any $v\in \mathfrak g -\mathfrak h$ the Lie subalgebras ${\mathfrak h}'=\mathfrak h \oplus \C v$, thus contradicting primitiveness and the fact that $\mathfrak h$ is supposed to have codimension $\geq 2$.
\end{remark}

\begin{remark}
    The obvious necessary condition for the existence of $\varphi$ is that $\mathfrak h$ is algebraic, that is, there exists one can realize $\mathfrak h$ as the Lie algebra of a Zariski closed subgroup $H$ of an algebraic group $G$ with Lie algebra $\mathfrak g$. This condition holds under the assumption of both theorems. Indeed, according to Remark \ref{R: heqnormh}, one has  $N_{\mathfrak g}(\mathfrak h) = \mathfrak h$ so that one can choose $H$ to be the stabilizer of $\mathfrak h$ with respect to the adjoint action $G\times \mathfrak g\to \mathfrak g$.
\end{remark}

\begin{remark}
    Consider the Lie algebra $\mathfrak{aff}=\mathfrak{gl}_n\ltimes \C^n$ of the group of affine motions of $\mathbb C^n$ ($n\geq 2$). This Lie algebra is complete and  $C_{\mathfrak{aff}}(D^1(\mathfrak{aff})) = \{0\}$. From these properties, one gets that $\mathfrak g = \mathfrak{aff} \oplus \mathfrak{aff}$ and $\mathfrak h = {\rm diag}(\mathfrak{aff}) \subset \mathfrak{aff} \oplus \mathfrak{aff}$ satisfy the hypothesis of Theorem  \ref{THM:A}. However, the subspace $\C^n\times \{0 \} + \mathfrak h$ is a non-trivial Lie subalgebra of $\mathfrak g$ containing $\mathfrak{h}$, therefore $(\mathfrak g, \mathfrak h)$ does not satisfy the hypothesis of Theorem \ref{THM:B}.
\end{remark}

\subsection{Lie algebras of rational vector fields}
A theorem by Deligne \cite[appendix B]{MR3719487} ensures that any finite-dimensional Lie algebra $\mathfrak g$ can be realized as a transitive Lie algebra of rational vector fields on a projective manifold $X$. In this Theorem, $\dim \mathfrak g = \dim X$.
On the other hand,  every projective manifold admits many transitive Lie algebras of rational vector fields. To see this, let $X$ be an arbitrary $n$-dimensional projective manifold.
First, embed it on a projective space $\mathbb P^N$ and then consider a linear projection
$\mathbb P^N \dashrightarrow \mathbb P^n$ to obtain a generically finite dominant rational map $\pi : X \dashrightarrow \mathbb P^n$.
The pull-back of vector fields from $\mathbb P^n$ to $X$ (well defined since we are working with rational vector fields and $\pi$ is generically finite) gives a Lie algebra morphism from the Lie algebra of vector fields on $\mathbb P^n$ to the Lie algebra of vector fields on $X$. Since $\mathbb P^n$ has many transitive Lie algebras of vector fields (e.g., the Lie algebra of holomorphic vector fields), the same holds for $X$.

There exist pairs $(\mathfrak g, \mathfrak h)$ not realizable as transitive Lie algebra of rational vector field on an algebraic variety. Some examples are known where the coefficients of vector fields must be rational in the variables and their exponentials. To prove that the exponential is the only transcendental function needed to realize any pairs as a Lie algebra of vector fields is known as the Lie Conjecture \cite{MR2067331}.

It is interesting to compare Theorem \ref{THM:A} with the discussion carried out there about the following
question: which pairs $(\mathfrak g, \mathfrak h)$ can be realized as Lie algebras of polynomial or rational vector fields?

At several points of our exposition, we will use results presented by Jan Draisma in the survey \cite{MR2902724}.

\section{Transitive and  complete Lie algebras}\label{S:transitive}
Throughout the text, we fix $n\ge1$, and we denote by  $\VF_n$ the (infinite-dimensional)
Lie algebra of germs of complex analytic vector fields at the origin of $\mathbb C^n$.
In this paragraph, we present some simple facts about transitive Lie subalgebras of $\VF_n$ and its formal counterpart $\VFR_n$, the Lie algebra of formal vector fields in $n$ variables. In these contexts, transitivity of a Lie subalgebra $ \mathfrak g \subset  \VF_n$ (or $\VFR_n$) means that the evaluation at $0$ from $\mathfrak g$ to $T_0 \mathbb C^n$ is onto.

A transitive Lie subalgebra $\mathfrak g\subset\VFR_n$ has an isotropy subalgebra $\mathfrak h$ at $0$  of codimension $n$, and the vector space $\mathfrak g/\mathfrak h$ is identified with the tangent space $T_0\mathbb C^n$ by the evaluation at $0$. Although our main result only involves finite-dimensional Lie algebra of rational vector fields, many of the transitive Lie algebras properties mentioned in this work
remain valid in the infinite-dimensional and formal case. A good account of their structure is given in  \cite{MR756031} to which we will refer at multiple places.

\begin{lemma}\label{L:hNoIdeal}
    Let $\mathfrak g\subset\VFR_n$ and $\mathfrak h \subset \mathfrak g$ be as above. Then $\mathfrak h$ contains no non-zero  $\mathfrak g$-ideal.
\end{lemma}
\begin{proof}
    Denote by $\M=(x_1,\ldots,x_n)$ the maximal ideal of $\C[[x_1,\ldots,x_n]]$, and define the order $\mathrm{ord} (w)$
    of an element $w \not=0$ of $\VFR_n$ as the smallest integer $m \geq 0$ such that the image of  $w$ in $\VFR_n/\M^{m+1} \VFR_n$
    is not zero. When $w\in\mathfrak h$, i.e $\mathrm{ord}(w)\geq 1$, there exists, by transitivity, an element $v\in\mathfrak g$
    such that $\mathrm{ord}([v,w])= \mathrm{ord}(w)-1$.
    Let now $\mathfrak k\subset\mathfrak h$ be a $\mathfrak g$-ideal, and, assuming non-trivial, let $w\in\mathfrak k$
    be of minimal order. Then $\mathfrak k$ should also contain $[v,w]$, what contradicts minimality.
\end{proof}

\subsection{Guillemin-Sternberg Theorem}\label{S:GS}
According to \cite[Theorem 3.3 and Remark 3.6]{MR2902724}, a version of Lemma \ref{L:hNoIdeal} led Guillemin-Sternberg to introduce the concept of effective primitive pair $(\mathfrak g,\mathfrak h)$ --- $\mathfrak g$ and $\mathfrak h$ are abstract Lie algebras with $\mathfrak h \subset \mathfrak g$ and $\mathfrak h$ contains no non-trivial $\mathfrak g$-ideal --- and to prove that effective primitive pairs can be realized
as Lie subalgebras of $\VFR_n$ where $n = \dim \mathfrak g - \dim \mathfrak h$  as stated in the result below.

\begin{thm}\label{T:GuilleminSternberg}
    Let $\mathfrak g$ be a Lie algebra over $\mathbb C$
    with a Lie subalgebra $\mathfrak h$ of finite codimension $n$. Then there exists a Lie algebra
    homomorphism $\phi:\mathfrak g\to\VFR_n$ such that $\mathfrak h$ is sent to
    the isotropy subalgebra at $0$. Moreover, the kernel $\mathfrak k=\ker(\phi)$ is the largest
    $\mathfrak g$-ideal contained in $\mathfrak h$, and the homomorphism $\phi$ is unique up to
    the action of $\widehat{ \mathrm{Diff}(\mathbb C^n,0)}$, i.e., the group of formal  diffeomorphisms of $(\mathbb C^n,0)$.
\end{thm}

In particular, $\phi$ is an embedding if, and only if  $(\mathfrak g,\mathfrak h)$ is an effective pair.

\subsection{The case $\mathfrak g$ has finite dimension}\label{S:GSfinite}
When $\mathfrak g$ has finite dimension, one can easily describe the construction of $\phi$ as follows.
Let $G$ be a Lie group with Lie algebra $\mathfrak g$. The isotropy subalgebra $\mathfrak h$ is the Lie algebra of
a (not necessarily closed) subgroup $H$ of $G$.
The orbits of the right action of $H$ on $G$ define a smooth foliation $\mathcal H$.
When $H$ is a closed subgroup of $G$, then the leaf space of $\mathcal H$
is naturally identified with the space of left $H$-cosets $G/H$.
Notice that the left action of $G$ on itself preserves the foliation $\mathcal H$.

Restrict now to a sufficiently small neighborhood $U$ of the identity in $G$. Since the foliation $\mathcal H$ is smooth,
we can consider the local leaf space of $\mathcal H$, well defined as a Hausdorff manifold of dimension $n=\textrm{dim}\ \mathfrak g -\textrm{dim}\ \mathfrak h$. Denote it by $U/ \mathcal H$. The Lie algebra of right invariant vector fields
on $G$ descends to a Lie algebra of vector fields on $U/\mathcal H$. Let us denote its image by $\mathfrak g_{{\rm right}}$.
We can identify the germ of $U/ \mathcal H$ at the identity with $(\mathbb C^n,0)$, and $\phi$ is the natural
morphism $\phi:\mathfrak g\to\mathfrak g_{{\rm right}}\subset\VF(U/ \mathcal H,\id)$ induced by the above construction.

We note that an element $g_0\in G$ sufficiently close to the identity will act trivially by left translations on $U/\mathcal H$ if,
and only if, for all $g\in G$ close enough to the identity, we have
\[
    gH=g_0 gH \ \left(=g\underbrace{(g^{-1}g_0g)}_{\in H} H\right)
\]
meaning that $g^{-1}g_0g$ belongs to $H$ for all $g$, i.e.,  that $g_0$ belongs to the intersection
\[
    K:=\bigcap_{g\in G} gHg^{-1}
\]
which is the largest normal subgroup in $G$ contained in $H$. Its Lie algebra $\mathfrak k=\mathrm{Lie}(K)$
is the largest $\mathfrak g$-ideal  contained in $\mathfrak h$.

\subsection{Relationship between invariant foliations  and intermediate subalgebras}\label{SS:corrfolal}
We follow the exposition given in \cite[pp. 3-4]{MR756031} from where we borrow notations.

Let $\mathfrak g\subset \VFR_n$ be a transitive Lie algebra of formal vector fields. By  a \textit{ codimension $m$ formal regular foliation}, we mean a subset  $\F$  of $\VFR_n$ which is both a Lie subalgebra and a free  $\C[[x_1,\ldots,x_n]]$-module of rank $n-m$ such that
\[
    T_0\F:= \{v(0)\arrowvert v\in \F \}
\]
has dimension $n-m$ over $\C$.

According to Frobenius theorem,  there exists a formal change of coordinates
\[
    (x_1,\ldots,x_n)\mapsto (\underbrace{ y_1,\ldots,y_m}_{:=y},\underbrace{ z_1,\ldots,z_{ n-m}}_{:=z})
\]
such that the foliation $\F$ is  generated over $\C[[x_1,\ldots,x_n]]$ by the vector fields $\dfrac{\partial}{\partial z_i}$
and $\mathcal P=\C[[y]]$ is the ring of first integrals of $\F$, i.e., 
$$
    \mathcal P=\{ f\in \C[[x_1,\ldots,x_n]]\arrowvert v(f)=0\ \textrm{for every}\ v\in\F\}
$$ 
which is canonically attached to $\F$.

Of course, one can define \textit{ a germ of regular convergent foliation} $\mathcal G$ along the same way, replacing $\VFR$ by $\VF$, $\C[[x_1,\ldots,x_n]]$ by the ring of convergent power series  $\C\{ x_1,\ldots,x_n\}$ and so on.

Suppose that $\F$ is now preserved by $\mathfrak g$,   i.e., $[\mathfrak g, \F]\subset \F$. The ring  $\mathcal P$ is then preserved under the action of $\mathfrak g$ by derivations. In terms of the previous coordinates, each element of $\mathfrak g$ expresses as
\[
    \xi=\sum_{j=1}^m \xi_j(y)\dfrac{\partial}{\partial y_j}+  \sum_{j=1}^{ n-m}\xi_{m+j}(y,z)\dfrac{\partial}{\partial z_j} \, .
\]
The vector fields of $\mathfrak g$ for which $\xi_1 (0)=\ldots=\xi_m (0)=0$ determine a Lie subalgebra  of $\mathfrak g$, namely the unique subalgebra $\mathfrak l\supset \mathfrak h$ such that $\mathfrak l (0)=T_0\F$. Note that one inherits a natural morphism of Lie algebras
\begin{align*}
    \varphi:	\mathfrak{l}  &\longrightarrow \VFR_{n-m}=\bigoplus_j\C[[z_1,\ldots,z_{n-m}]]\dfrac{\partial}{\partial z_j} \\
    \xi  & \mapsto  \sum_{j=1}^ { n-m}\xi_{m+j}(0,z)\dfrac{\partial}{\partial z_j}
\end{align*}
whose image is nothing but the restriction of $\mathfrak l$ to the leaf through $0$ defined by $\{y=0\}$. Identifying  $\VFR_{n-m}$ with a subalgebra of $\VFR_n$ (depending on the choice of $z$), any element of $\mathfrak g$ takes the form
\[
    \xi=\sum_{j=1}^m \xi_j(y)\dfrac{\partial}{\partial y_j} + \sum_{\alpha \in \N^m} y^\alpha \eta_\alpha
\]
with the usual multi-index notation and $\eta_\alpha = \sum \frac{1}{\alpha !}\frac{\partial ^\alpha \xi_{m+j}}{\partial y^\alpha}(0,z) \frac{\partial}{\partial z_j}\in \VFR_{n-m}$. Moreover, up to rechoosing the coordinates $z_1,\ldots,z_{ n-m}$, one can assume that the $\eta_\alpha$'s lie in $\overline{\varphi(\mathfrak l)}$,  the closure of $\varphi({\mathfrak l})$ with respect to the Krull topology (see  \cite[p.4 and Corollary 2.1]{MR756031} ). In particular, any formal vector field  in $\VFR_{n-m}\subset \VFR_n$ belonging to the centralizer of $\varphi(\mathfrak l)$ in $\VFR_{n-m}$ also belongs to the centralizer of $\mathfrak g$ in  $\VFR_n$ (this property will be used in the proof of Proposition \ref{P:Cg}).

It turns out that this process can be reversed. The correspondence $\F\mapsto { \mathfrak l}(\F)$ defined above between $\mathfrak g$-invariant foliations $\F$ and  subalgebra $\mathfrak l$ of $\mathfrak g$ containing $\mathfrak h$ is one to one and onto (see \cite[Theorem 1.3]{MR756031}). Moreover, if $\mathfrak g\subset \VF_n$, and $\mathfrak l={\mathfrak l}(\F)$, then the foliation $\F$ is clearly convergent, i.e $\F= \G\otimes \C[[x_1,\ldots,x_n]]$ where $\G$ is convergent in the sense specified above. Indeed, let $E\subset \mathfrak g$ a $n$-dimensional complementary subspace to  $\mathfrak h$. Because of the $\mathfrak g$ invariance $[\mathfrak g,\F]\subset \F$, $\F$ necessarily coincides with the germ of distribution obtained by propagating the vector subspace $\mathfrak l(0)$ from the origin to a full neighborhood using the local flows of vector fields $v\in E$. Or, if one prefers, $\F$ is given by parallel transport of $\mathfrak l(0)$ along the germs of analytic curves $\exp{tv}(0)$ with respect to the unique connection $\nabla$ on $\VF_n$, such that $\nabla_v X=[v,X]$.

\subsection{Relationship between $N_{\mathfrak g}(\mathfrak h)$, $C_\VF(\mathfrak g)$ and zeroes of $\mathfrak h$.}\label{S:NghZeroh}
In this subsection, we will consider the case of primitive pairs $(\mathfrak g,\mathfrak h)$, i.e., there is no intermediate Lie algebra between $\mathfrak h$ and $\mathfrak g$.
Let $I\subset \C[[x_1,\ldots,x_n]]$  the ideal generated by the coefficients of elements of $\mathfrak h$. Denote by $\mathcal M$ the maximal ideal of the local Noetherian ring $A= \C[[x_1,\ldots,x_n]]/I$.  We will use the notation $T_0 (Z(\mathfrak{h}))$  for the Zariski tangent space $ { (\M/ { \M^2})}^*\subset T_0 \C^n$ of the subscheme $Z(\mathfrak{h})=\mathrm{Spec} A$ at the closed point.

An obstruction to primitiveness is given by $N_{\mathfrak g}(\mathfrak h)$, which can be strictly larger than $\mathfrak h$. The following lemma characterizes this case.

\begin{lemma}\label{L:NghZeroh}
    Let $\mathfrak g\subset \VFR_n$ be a transitive Lie subalgebra
    and $\mathfrak h\subset\mathfrak g$ be the isotropy Lie subalgebra of $\mathfrak g$ at $0$.
    The  schematic zero locus $Z(\mathfrak h)$ of $\mathfrak h$ is smooth at the closed point, and is the leaf through the origin of a regular formal $\mathfrak g$-invariant foliation $\mathcal Z$ on $(\mathbb C^n,0)$. This foliation is convergent whenever $\mathfrak g \subset \VF_n$. Moreover, the tangent space $T_0\ Z(\mathfrak{h})$ of $Z(\mathfrak h)$ at $0$ is given by the exact sequence of vector spaces
    \[
        0\to \mathfrak h \to N_{\mathfrak g}(\mathfrak h) \xrightarrow{ev_1} T_0\ Z(\mathfrak{h})=T_0\mathcal Z \to 0
    \]
    where the right arrow is evaluation $ev_1:v\mapsto v(0)$. In particular, $N_{\mathfrak g}(\mathfrak h)=\mathfrak h$
    if, and only if, $\mathfrak h$ has isolated zero set.
\end{lemma}
\begin{proof}
    Let $v\in \mathfrak g$. From the Lie algebra structure, the following properties are obviously equivalent:
    \begin{enumerate}
        \item $v\in N_{\mathfrak g}(\mathfrak h)$
        \item $v(I)	\subset I$
        \item For all $f\in I$, $v(f)(0)=0$.
    \end{enumerate}
    If one combines these observations with the transitivity of $\mathfrak g$, then one obtains the exact sequence
    \[
        0\to \mathfrak h \to N_{\mathfrak g}(\mathfrak h) \to T_0 Z(\mathfrak{h}) \to 0
    \]

    It remains to show that $Z(\mathfrak{h})$ is smooth, or in other words, that $A$ is regular. By classical properties of Noetherian local rings, this amounts to prove that the associated graded ring $\mathrm{Gr}_\M (A)=\bigoplus_{i=0}^{+\infty} \M^i/{\M}^{i+1}$ is isomorphic to the $\C$-algebra $\C[X_1,\ldots,X_d]$ of polynomials with $d$ indeterminates where $d$ is the dimension of the $\C$ linear space $T_0   Z(\mathfrak{h})$. To do this, note firstly that any vector field $v\in \VFR_n$ such that $v(I)\subset I$ induces a derivation on $A$, hence a derivation $D$ on $\mathrm{Gr}_\M (A)$ which is a $\C$-linear operator of degree $-1$ since
    \[
        D( \M^i/{\M}^{i+1})\subset \M^{ i-1}/{\M}^{i}.
    \]
    Now, let us choose vector fields $v_i\in N_{\mathfrak g}(\mathfrak h), i=1, \ldots,d$  such that $(v_i(0), i=1,\ldots,d)$ form a basis  of $T_0(Z(\mathfrak{h}))$.   Up to performing a linear change of coordinates in $\C[[x_1,\ldots,x_n]]$, one can assume that $v_i(x_j)(0)=\delta_{ij}$, $i=1,\ldots,d$.  Nakayama's Lemma implies that the $x_j$'s modulo $I$ generate $\M$.  Let $D_i$ be the associate derivation of $v_i$ on $\mathrm{Gr}_\M (A)$ (well defined according to the three equivalent properties listed above). Observe that $D_i(x_j)=\delta_{ij}$, where the $x_j, j=1, \ldots,d$ are regarded as homogeneous elements of degree $1$ in  $\mathrm{Gr}_\M (A)$. Consequently, the natural surjective morphism of graded rings $\C[X_1,\ldots,X_d]\mapsto \mathrm{Gr}_\M (A)$ induced by the evaluation  $P\mapsto P(x_1,\ldots,x_d)$ is indeed an isomorphism. Actually, it is an isomorphism of differential rings with respect to the set of derivations $ \{ \frac{\partial}{\partial X_i} \}$ and $\{ D_i \}$. This proves the regularity of $A$, as wanted.

    Because $N_\mathfrak{g}(\mathfrak{h}) \cdot I\subset I$ and from the correspondance foliation-subalgebra recalled in Subsection \ref{SS:corrfolal}, note also that $Z(\mathfrak h)$ is nothing but the leaf through $0$ of the foliation $\mathcal Z$ associated to $N_\mathfrak{g}(\mathfrak{h})$. In more geometric terms, $Z(\mathfrak h)$ is the "orbit" of $N_\mathfrak{g}(\mathfrak{h})$ through the origin. According to Subsection \ref{SS:corrfolal}, $\mathcal Z$ is convergent if $\mathfrak g$ is so and the foliation $\mathcal Z$  is obtained by translation of this orbit as a particular case, with $\mathfrak l= N_\mathfrak{g}(\mathfrak{h})$, of the $\mathfrak g$ invariant germ of  foliation constructed in Subsection \ref{SS:Prim}.
\end{proof}

We will denote by
\[
    C_{ \VFR_n}(\mathfrak g):=\{ v \in \VF \, \vert \, [v,\mathfrak g]=0 \}
\]
the centralizer of $\mathfrak g$ in $\VFR_n$.

\begin{lemma}\label{L:BoundCentralizer}
    Let $\mathfrak g\subset \VFR_n$ be transitive
    and $\mathfrak h\subset\mathfrak g$ be its isotropy Lie subalgebra at $0$.
    Then the evaluation map defines an embedding:
    \begin{align*}
        ev_2:C_{ \VFR_n}(\mathfrak g) &\hookrightarrow T_0\mathcal Z  \\
        v & \mapsto v(0)
    \end{align*}
    where $\mathcal Z$ is as in Lemma \ref{L:NghZeroh}. In particular,
    if $N_{\mathfrak g}(\mathfrak h)=\mathfrak h$, or equivalently $\mathfrak h$ has isolated zero set, then $C_{ \VFR_n}(\mathfrak g)$ is trivial (and $\mathfrak g$ is centerless).
\end{lemma}
\begin{proof}
    As in the proof of Lemma \ref{L:NghZeroh}, the action of $C_{ \VFR_n}(\mathfrak g)$ must preserve
    $\mathfrak h$, as it commutes with it. Therefore, the action of $C_{ \VFR_n}(\mathfrak g)$ must also preserve the zero set of $\mathfrak h$.
    We deduce that the evaluation map sends $C_{ \VFR_n}(\mathfrak g)$ into $T_0\mathcal Z$. It remains to show that its kernel is trivial.
    For this, consider an element $w\in  C_{ \VFR_n}(\mathfrak g)$ such that $w(0)=0$. Assume for a while that $w$ does not vanish identically. Consider the positive integer $\mathrm{ord } (w)$ as defined in Lemma \ref{L:hNoIdeal}. By transitivity, there exists $v\in \mathfrak g$ such that $\mathrm{ord }([v,w])= \mathrm{ord }(w)-1$. On the other hand $[v,w] =0$: absurd.
\end{proof}

\subsection{Centralizers of transitive Lie algebras of formal vector fields}
Our next result shows that the map of Lemma \ref{L:BoundCentralizer} is an isomorphism. In this way, we
obtain a description of centralizers of transitive Lie algebras of formal vector fields.

\begin{prop}\label{P:Cg}
    Let $\mathfrak g \subset \VFR_n$ be a transitive subalgebra
    and $\mathfrak h\subset\mathfrak g$ be the isotropy Lie subalgebra at $0$.
    Then we have isomorphisms
    \[
        C_{\VFR_n}(\mathfrak g) \simeq T_0\mathcal Z\simeq \frac{N_{\mathfrak g}(\mathfrak h)}{\mathfrak h} \, .
    \]
    Moreover,  $C_{\VFR_n}(\mathfrak g)\subset\VF_n$ if $ \mathfrak g\subset \VF_n$.
\end{prop}
\begin{proof}
    Because of Lemmas \ref{L:NghZeroh} and \ref{L:BoundCentralizer}, it only remains to prove
    that $\frac{N_{\mathfrak g}(\mathfrak h)}{\mathfrak h}$ can be embedded in $C_{\VFR_n}(\mathfrak g)$
    (then dimensions will coincide, making the morphism of Lemma \ref{L:BoundCentralizer} an isomorphism).
    In view of this, let us first assume that $\mathfrak g$ is finite-dimensional. According to Theorem \ref{T:GuilleminSternberg}, we can suppose that $\mathfrak g$ is convergent and its local action 
    on $( \C^n,0)$ is given in Section  \ref{S:GSfinite}:  the right action of $H$ defines a foliation $\mathcal H$, and
    the infinitesimal action of $\mathfrak g$ on $(\mathbb C^n,0)$ identifies with the action of
    $\mathfrak g_{{\rm right}}$ on the quotient $U/ \mathcal H$ of a neighborhood $U$ of the identity in $G$.

    The right action of $G$ on itself does not preserve the foliation $\mathcal H$ in general.
    If an element $g \in G$  maps left $H$-cosets to left cosets through right multiplication, then it must map the coset $g H$ to $H$ since the set $g H g^{-1}$ contains the identity.
    It follows that the set of  elements in $G$ preserving $\mathcal H$ through right multiplication
    is formed precisely by $N_G(H)$  the normalizer of $H$ in $G$.

    From the infinitesimal point of view, left-invariant vector fields are not, in general, infinitesimal automorphisms
    of $\mathcal H$. Only those in the normalizer of $\mathfrak h$ in $\mathfrak g$ will be infinitesimal automorphisms and will
    descend to the leaf space $U/\mathcal H$.  The left-invariant vector fields in $\mathfrak h$ descend to zero, and the quotient
    $N_\mathfrak{g}(\mathfrak h)/{\mathfrak h}$ injects into the centralizer of $\mathfrak g_{{\rm right}}$ on $U/\mathcal H$.

    This conclusion still holds in general. Indeed, from Lemma \ref{L:NghZeroh},   one remarks that $\mathfrak h$ is the kernel of the surjective Lie algebra morphism induced by restriction: $N_{\mathfrak g}(\mathfrak h)\mapsto {N_{\mathfrak g}(\mathfrak h)}_{|Z(\mathfrak h)}$. In particular,  $\tilde N:={N_{\mathfrak g}(\mathfrak h)}_{|Z(\mathfrak h)}$ is transitive with trivial isotropy. Indeed, $Z(\mathfrak h)$ coincides with the leaf through the origin of the foliation associated to $ N_{\mathfrak g}(\mathfrak h)$ according to the observations made in the proof of Lemma \ref{L:NghZeroh} and Subsection \ref{SS:corrfolal}.
    By Theorem \ref{T:GuilleminSternberg}, $\tilde N$ is formally conjugated to the Lie algebra of right invariant vector fields of a Lie group $G$ near $e_G$. Its centralizer $C(\tilde N)$ in the Lie algebra of formal vector fields tangent to $Z(\mathfrak h)$  thus corresponds by this isomorphism to the left-invariant vector fields on $(G,e_G)$ and then is a transitive subalgebra (of the Lie algebra of formal vector fields on $Z(\mathfrak h)$)  with trivial isotropy, isomorphic as a Lie algebra to  $\tilde N$. According to the remarks given in Subsection \ref{SS:corrfolal}, each element of $C(\tilde N)$ extend as a vector field lying in $\VFR_n$ and commuting with $\mathfrak g$. One obtains in this way the sought isomorphism. As a result of the above description, and according to the classical relationship between right-left invariant vector fields on a Lie group $G$ (namely the structure of opposite Lie algebras induced on $T_e G$), we see that one inherits a natural isomorphism of Lie algebra
    \begin{align*}
        \frac{N_{\mathfrak g}(\mathfrak h)}{\mathfrak h}  &\longrightarrow C_{\VFR_n}(\mathfrak g) \\
        v  & \mapsto  -{ev_2}^{-1}(v (0))
    \end{align*}

    Retaining the notations and arguments of Subsection \ref{SS:corrfolal}, one can  conclude too that  $C_{\VFR_n}(\mathfrak g)$ is analytic  if  $\mathfrak g$ is so.
\end{proof}

\subsection{Complete Lie algebra}\label{S:Complete}
Recall some basic notions of \cite[Chapter I, Sections 2-3]{MR559927}.
A derivation of a Lie algebra $\mathfrak g$ is a morphism $\partial:\mathfrak g\to \mathfrak g$
of $\mathbb C$-vector spaces satisfying the Leibniz rule $\partial[v,w]=[\partial v,w]+[v,\partial w]$. The set
$\mathrm{Der}(\mathfrak g)$ of derivations forms a Lie algebra with respect to the Lie bracket $\{\partial_1,\partial_2\}=\partial_1\partial_2-\partial_2\partial_1$.

A natural example of derivation of a Lie algebra $\mathfrak g$ is provided by the adjoint action of elements of $\mathfrak g$:
\begin{align*}
    \mathrm{ad}_v:\mathfrak g & \to\mathfrak g \\
    w & \mapsto [v,w].
\end{align*}
Jacobi identity for $\mathfrak g$ implies the Leibniz rule
$$
    \mathrm{ad}_v[w_1,w_2]=[\mathrm{ad}_vw_1,w_2]+[w_1,\mathrm{ad}_vw_2]
$$
and also the identity
$$
    \{\mathrm{ad}_{v_1},\mathrm{ad}_{v_2}\}
    =\mathrm{ad}_{v_1}\mathrm{ad}_{v_2}-\mathrm{ad}_{v_2}\mathrm{ad}_{v_1}    =\mathrm{ad}_{[v_1,v_2]}.
$$
We, therefore, have an exact sequence
$$
    0 \to C_{\mathfrak g}(\mathfrak g) \to \mathfrak g\stackrel{\mathrm{ad}}{\longrightarrow} \mathrm{Der}(\mathfrak g)
$$
where $C_{\mathfrak g}(\mathfrak g)$ is the center of $\mathfrak{g}$.
The image in $\mathrm{Der}(\mathfrak g)$ is the subalgebra of \textit{inner derivations}. We say that $\mathfrak g$ is \textit{complete}
when the morphism $\mathrm{ad}$ above is an isomorphism:
$$
    \mathfrak g\stackrel{\sim}{\longrightarrow} \mathrm{Der}(\mathfrak g)
$$
i.e., when $\mathfrak g$ is centerless ($C_{\mathfrak g}(\mathfrak g) =0 $), and  all its derivations are inner.

When $\mathfrak g\subset\VF_n$, we can have more derivations by considering the normalizer
\[
    N_{\VF_n}(\mathfrak g)  = \{ v \in \VF_n \, \vert \, [v,\mathfrak g]\subset \mathfrak g \}.
\]
Indeed, $\mathfrak g$ is stabilized by the adjoint action of $N_{\VF_n}(\mathfrak g)$ on $\VF_n$, and we have

\begin{lemma}\label{L:exact}
    Let $\mathfrak g \subset \VF_n$ be a Lie algebra of germs of analytic vector fields. Then, the normalizer $N_{\VF_n}(\mathfrak g)$ fits
    into the exact sequence
    \[
        0 \to C_{\VF_n}(\mathfrak g) \longrightarrow N_{\VF_n}(\mathfrak g) \stackrel{\mathrm{ad}}{\longrightarrow} \Der(\mathfrak g)
    \]
    where $C_{\VF_n}(\mathfrak g) = \{ v \in \VF_n \, \vert \, [v,\mathfrak g]=0\}$ is the centralizer of $\mathfrak g$ in $\VF_n$,
    and $\Der(\mathfrak g)$ is the  algebra of derivations of $\mathfrak g$.
\end{lemma}
\begin{proof}
    One has just to observe that for any element $v \in N_{\VF_n}(\mathfrak g)$ the restriction of $\ad_v(\cdot) = [v,\cdot]$
    to $\mathfrak g$  is a derivation of $\mathfrak g$ since
    $[v,\mathfrak g] = \ad_v(\mathfrak g) \subset \mathfrak g$. The kernel of the morphism $v \mapsto (\ad_v)_{|\mathfrak g}$ is, by definition, $C_{\VF_n}(\mathfrak g)$.
\end{proof}

Of course, one can rewrite these left exact sequences in the formal setting, replacing $\VF_n$ with $\VFR_n$.

\subsection{The main ingredient for the proof of  Theorem \ref{THM:A}}\label{S:ingredient}
As already mentioned in the Introduction, Theorem \ref{THM:A} is a straightforward consequence of  \cite[Theorem 5.12]{MR3719487}.
If $X$ is a complex manifold and $p \in X$ is a  point, we will denote the Lie algebra of germs of analytic vector fields at $p$ by $\VF(X,p)$ or simply by $\VF$.

\begin{thm}{\cite[Theorem 5.12]{MR3719487}}\label{T:BC}
    Let $X$ be a smooth irreducible complex algebraic variety and let $\mathfrak g$ be a
    finite-dimensional Lie subalgebra of rational vector fields on $X$. For a general $p \in X$, $\VF=\VF(X,p)$ and $\mathfrak h$ is the isotropy subalgebra at $p$. If $\mathfrak g$ is transitive, $N_{\VF}(\mathfrak g) = \mathfrak g $, $C_{\mathfrak g}(\mathfrak g) = 0 $ and $N_{\mathfrak g}(\mathfrak h) = \mathfrak h $,  then there exists a dominant rational map $\bar f: X \dashrightarrow G/H$ such that $\bar f_* \mathfrak g$ coincides with the Lie algebra of fundamental  vector fields on $G/H$.
\end{thm}

This theorem is proved in several steps.
\begin{enumerate}
    \item A general point $p \in X$ and a local analytic chart $(X,p) \to (\mathbb C^n,0)$ are fixed. In this chart $\mathfrak g$ is isomorphic to $\mathfrak g_0 \subset \VF_n$. One considers
        \[
            V= \{\varphi : (\mathbb C^n,0) \to X \textrm{ germs of biholomorphisms. s.t. } \varphi^\ast \mathfrak g = \mathfrak g_0\}
        \]
    Under the hypothesis $\mathfrak g$ is transitive and $N_{\VF}(\mathfrak g) = \mathfrak g $ it is proved that $V$ is an algebraic variety and $\dim V = \dim \mathfrak g$, the evaluation map  $\pi : V \to X ; \varphi \mapsto \varphi(0)$ being dominant.

    \item From the definition of $V$, both Lie algebras $\mathfrak g_0$ and $\mathfrak g$ can be prolonged as $\widetilde{g}_0$ and $\widetilde{g}$, two commuting parallelisms on $V$   ({\it, i.e.} transitive Lie algebra of rational vector fields of dimension equal to $\dim V$). Morevover $\widetilde{h}_0$ generates $\ker d \pi$ and $\widetilde {\mathfrak g}$ is $\pi$-projectable on $\mathfrak g$.

    \item Step (2)  together with $C_{\mathfrak g}(\mathfrak g) = 0 $ implies that $\mathfrak g = {\rm Lie}(G)$, $G$ being an algebraic group and implies the existence of a rational map $f : V \to G$  such that  $f^\ast {\rm Lie}_\text{left}(G) =  \widetilde{\mathfrak g}_0$ and  $f^\ast {\rm Lie}_\text{right}(G) =  \widetilde{\mathfrak g}$ where ${\rm Lie}_\text{left}(G)$ (resp.  ${\rm Lie}_\text{right}(G)$) is the Lie algebra of left (resp. right) invariant vector fields of $G$.

    \item A fiber of $\pi : V \to X$ is the union of orbits of $\widetilde{\mathfrak h}_0$. The image of these orbits by $f$ are orbits by the right translation of a connected algebraic subgroup $H$ of $G$, and $f$ induces a rational map $\bar f : X \to G/H$, which is the sought rational map.
\end{enumerate}
We refer to \cite{MR3719487} for details.

\subsection{Proof of Theorem \ref{THM:A}}
Consider now $\mathfrak g\subset \VF(X,p)$ and $\mathfrak h$ satisfying assumptions of Theorem \ref{THM:A}, that is  $\mathfrak g$ is transitive, complete, with isotropy subalgebra $\mathfrak h$, and $N_{\mathfrak g} (\mathfrak h)=\mathfrak h$.

Since $\mathfrak g$ is complete, we have an isomorphism $\Der(\mathfrak g)\simeq\mathfrak g$ induced by  the adjoint action of $\mathfrak g$.
Moreover, we derive from Lemma \ref{L:exact} the exact sequence
\[
    0 \to C_{\VF}(\mathfrak g) \longrightarrow N_{\VF}(\mathfrak g) \stackrel{\mathrm{ad}}{\longrightarrow} \Der(\mathfrak g)\to 0
\]
as the restriction of $\mathrm{ad}$ to $\mathfrak g\subset N_{\VF}(\mathfrak g)$ is onto. On the other hand,
since $N_{\mathfrak g} (\mathfrak h)=\mathfrak h$, we deduce from Proposition \ref{P:Cg} that $C_{\VF}(\mathfrak g)=\{0\}$; therefore $C_{\mathfrak g}(\mathfrak g)=\{0\}$ and  (from the above sequence)
\[
    N_{\VF}(\mathfrak g)\simeq \Der(\mathfrak g)\simeq \mathfrak g.
\]
We can apply Theorem \ref{T:BC} to conclude the proof of Theorem \ref{THM:A}. \qed

The proof shows that  the hypothesis of Theorems  \ref{THM:A} and \ref{T:BC} differ by the assumption on the outer derivations of the Lie algebra $\mathfrak g$. In Theorem  \ref{T:BC},  $N_{\VF}(\mathfrak g) = \mathfrak g$ is an hypothesis on the pair $(\mathfrak g, \mathfrak h)$ : there is no outer derivations of $\mathfrak g$ given by the Lie bracket with a vector field, in a realisation of the pair. But in Theorem \ref{THM:A}, the completeness of $\mathfrak g$ is an hypothesis on the abstract Lie algebra independently from the realisation as Lie algebra of vector fields.

\section{Primitive Lie algebras}\label{S:primitive}

\subsection{Structure of primitive Lie Algebras [after Morozov]}
Let $\mathfrak g$ be a primitive Lie subalgebra of $\VFR_n$.
Equivalently, invoking Guillemin-Sternberg (see Section \ref{S:GS}),
we obtain an effective pair $(\mathfrak g,\mathfrak h)$, i.e., $\mathfrak h$ has finite codimension $n$,
contains no non-zero ideal and is maximal among proper Lie subalgebras of $\mathfrak g$.
In \cite[Section 4]{MR2902724} the following result by Morozov clarifies the structure of finite-dimensional effective primitive pairs.

\begin{thm}\label{T:Morozov}
    Suppose that $(\mathfrak g, \mathfrak h)$ is an effective primitive pair where $\mathfrak g$ is finite-dimensional.
    Then either $\mathfrak g$ is simple, or else we are in one of the following two situations:
    \begin{enumerate}
        \item $\mathfrak g$ is the direct sum $\mathfrak l \oplus \mathfrak l$  of two isomorphic simple Lie algebras
                and $\mathfrak h$ is the diagonal subalgebra; or
        \item $\mathfrak g$ is the semi-direct product $\mathfrak h \ltimes \mathfrak m$ where  $\mathfrak h$ is the direct sum of a semisimple Lie algebra with a Lie algebra of dimension at most $1$, and $\mathfrak m$ is an irreducible and faithful $\mathfrak h$-module               equipped with trivial Lie bracket: $[\mathfrak m, \mathfrak m] = 0$.
    \end{enumerate}
    Conversely, in the latter two cases, $(\mathfrak g,  \mathfrak h)$ is primitive and effective.
\end{thm}

The primitive Lie algebras fitting the description given by Item (1), respectively Item (2), of Morozov's Theorem, will be
called primitive Lie algebras of diagonal and affine type, respectively.  The primitive Lie algebras with
$\mathfrak g$ simple will be called simple primitive Lie algebras.

\subsection{The case of curves}\label{SS:curves}
The following result is extracted from the proof of  \cite[Theorem 6.5]{MR3842065}.

\begin{prop}\label{P:rationalLie}
    Let $C$ be a projective curve and $\mathfrak g$ a finite-dimensional Lie algebra of rational vector fields on $C$.
    If $\dim \mathfrak g \ge 2$, then there exists a morphism $\varphi: C \to \mathbb P^1$ such that $\mathfrak g$
	coincides with a subalgebra of  the pull-back under $\varphi$ of the Lie algebra of holomorphic vector fields on
    $\mathbb P^1$, i.e., $\mathfrak g \subset \varphi^* \mathfrak{sl}(2, \mathbb C)$.
\end{prop}
\begin{proof}
    A classical result of Lie (cf. \cite{MR2902724}) says that a finite-dimensional Lie subalgebra of
    $\mathbb C(z) \frac{\partial}{\partial z}$ has dimension at most three. Moreover, if its dimension is two then it is
    isomorphic to the affine Lie algebra $\mathfrak{aff}(\mathbb C)$, and if its dimension is three, then
    it its isomorphic to the projective Lie algebra $\mathfrak{sl}(2,\mathbb C)$.
	Therefore $\mathfrak g$ has dimension at most three.
			
    In both cases there exists $v_1, v_2 \in \mathfrak g$ satisfying $[v_1,v_2] = v_1$. Consider the morphism $\varphi : C \to \mathbb P^1$
	defined by the quotient $-\frac{v_2}{v_1}$.
			
	At an arbitrary point of $C$, choose a local analytic coordinate $w$. We can write
	$v_1 =a(w)  \frac{\partial}{\partial w}$,
    $v_2=b(w) \frac{\partial}{\partial w}$, $\varphi(w)= -b(w)/a(w)$ locally.
    On the one hand, the relation $[v_1,v_2]=v_1$ implies $ab' -ba' = a.$
	On the other hand,
	\[
	   \varphi^*  \frac{\partial}{\partial z} = -\frac{1}{(b/a)' (w)}  \frac{\partial}{\partial w} =
            -\frac{a^2}{(a'b-ab')}  \frac{\partial}{\partial w} = a  \frac{\partial}{\partial w} = v_1 \, .
	\]
	Similarly $\varphi^*  z \frac{\partial}{\partial z} = v_2$. This proves the proposition when
    $\mathfrak g\simeq \mathfrak{aff}(\mathbb C)$.
			
	To conclude the proof when $\mathfrak g \simeq \mathfrak{sl}(2,\mathbb C)$  it suffices to verify that
	$\tilde{v}_3 = \varphi^* z^2\frac{\partial}{\partial z}$ belongs to $\mathfrak g$. Let $v_3 \in \mathfrak{g}$ be such that
	$[v_1,v_3]=2 v_2$. Clearly $[v_1, \tilde{v}_3]$ is also equal to $2v_2$. Thus $[v_1, v_3 - \tilde{v}_3] = 0$ and $\tilde{v}_3$ must be
	linear combination with constant coefficients of $v_1$ and $v_3$.
\end{proof}

\subsection{Algebraic normalization of primitive Lie algebras of affine type}
A natural generalization of the argument used to prove Proposition \ref{P:rationalLie} gives the algebraic normalization of primitive Lie algebras of affine type.

\begin{prop}\label{P:affine}
    Let  $\mathfrak g$ be a non-abelian primitive Lie algebra of rational vector fields on a projective manifold $X$ of dimension $n$.
    If $\mathfrak g$ is of affine type then there exists a rational map $\varphi : X \dashrightarrow \mathbb C^n$
	which conjugates $\mathfrak g$ to a  Lie algebra of polynomial vector fields of degree at most one.
\end{prop}
\begin{proof}
    As  case $n=1$ has already been treated in Proposition \ref{P:rationalLie}, we will assume $n\ge 2$.
	Since $\mathfrak g$ is of affine type, we can write $\mathfrak g = \mathfrak h \ltimes \mathfrak m$ with
	$\left[ \mathfrak m, \mathfrak m \right] =0$. Near a general point $p$ of $X$ and up to isomorphism of $\mathfrak g$,
    one can suppose that $\mathfrak h$ is the isotropy algebra ${ \mathfrak h}_p$ and $\mathfrak m$ is generated over $\C$ by $\frac{\partial}{\partial x_1}, \ldots, \frac{\partial}{\partial x_n}$ in suitable local analytic coordinates $(x_1,\ldots,x_n)$.
	
    In order to establish the proposition, we are going to produce a $\mathbb C$-vector space $V \subset \mathbb C(X)$
    of dimension $n +1$, containing the constants $\mathbb C$, and invariant by the action of $\mathfrak g$ by derivations.
	
    To construct $V$ let us consider $\Theta_0 \in \det T_X \otimes \mathbb C(X)$ defined by the determinant of $\mathfrak m$, i.e., if $v_1,\ldots, v_n$ is a basis for $\mathfrak m$ then $\Theta_0 = v_1\wedge \cdots \wedge v_n$.
    Notice that for any $v \in \mathfrak g$,
    \begin{equation}\label{E:Trace}
        \mathcal L_v \Theta_0= \Tr(\ad(v)) \Theta_0
    \end{equation}
    where $\mathcal L_v$ is the Lie derivative with respect to $v$ (acting on multivector fields), $\ad(v) \in \End(\mathfrak m)$ is the endomorphism $\ad(v)(m) = [v,m]$, and $\Tr$ is the trace of an endomorphism.
	
    Let us consider the $\mathfrak g$-module $W = \wedge^{n-1} \mathfrak m \otimes \mathfrak g$. The action of $\mathfrak g$ on $W$ is obtained by distributing its action on factors employing Leibniz's rule, that is
    \[
        v \cdot ( \theta\otimes w):= \mathcal L_v\theta\otimes w +\theta\otimes [v,w].
    \]
    Consider the $\C$-linear map
    \begin{align*}
        \psi: W &\longrightarrow \mathbb C(X) \\
        \theta\otimes w &\mapsto \frac{\theta\wedge w} {\Theta_0}
    \end{align*}
    and let $V$ denote its image.
	One can, in addition, regard $\psi$ as a morphism of $\mathfrak g$-modules
    if one endows $\C(X)$ with the structure of $\mathfrak g$-module defined by
	\[
        v\cdot f:= v(f)+\Tr(\ad v)f
    \]
	where $v(f)$ is the usual derivative with respect to $v$.
	In particular, this implies that the image $V$ of $W$ is stable by $\mathfrak g$-derivations.
    Observe also that the splitting (as a vector space) $\mathfrak g= \mathfrak h\oplus \mathfrak m$ induces the splitting
	\[
        V:=\psi(W)= \C\oplus V_0
    \]
	and that  $\C= \psi ( \wedge^{n-1} \mathfrak m \otimes \mathfrak m$), $V_0=\psi ( \wedge^{n-1} \mathfrak m \otimes \mathfrak h)$.
	Indeed, non-zero functions in $V_0$ are non-constant as they vanish at the point $p$ fixed by the isotropy subalgebra  $\mathfrak h$.
    Remark finally that $V_0$ is not reduced to $\{ 0\}$, as $\mathfrak h\not =\{0\}$ (non-abelian case).
	
	Consider the action of $\mathfrak g$  by derivations
	on $V$.
	The subspace $V_0$   is invariant by the action of $\mathfrak h$ and is mapped to $\mathbb C$ by $\mathfrak m$.
    In particular, in the local coordinate chart defined above, the elements of $V_0$ are nothing but linear forms in the variables $( x_1,\ldots,x_n)$.
	
	Let us prove that $\dim V_0 = n$.
    On the one hand, the dimension of $V_0$  is at most $n$ as an immediate consequence of the previous description.
	On the other hand, the dimension of $V_0$ is at least $n$ as otherwise, the functions in $V_0$ would define a
    foliation of positive dimension $=n-\dim_\C V_0$  invariant by the Lie algebra $\mathfrak g$, contradicting the primitiveness of $\mathfrak g$.
	
	It is now clear that the rational map $\varphi : X \dashrightarrow \mathbb C^n$ with entries given by a basis of $V_0$ sends
	$\mathfrak g = \mathfrak h \ltimes \mathfrak m$ to a Lie algebra of vector fields of degree at most one. 	
\end{proof}

\subsection{Proof of Theorem \ref{THM:B}}
    Let $\mathfrak g$ be a non-abelian primitive Lie algebra with isotropy $\mathfrak h$.
    By the observations made in Subsection \ref{SS:curves}, one can assume that $\mathfrak h$ has codimension at least two in $\mathfrak g$. This easily implies that  $N_{\mathfrak g}(\mathfrak h) = \mathfrak h$.
    Otherwise, $N_{\mathfrak g}(\mathfrak h)$ would provide either an intermediate Lie algebra between $\mathfrak g$ and $\mathfrak h$, or $N_{\mathfrak g}(\mathfrak h)=\mathfrak g$. However, this latter case is excluded by effectiveness.

    Thus if $\mathfrak g$  happens to be a complete Lie algebra, then it satisfies the assumptions of Theorem \ref{THM:A} and the result follows. Theorem  \ref{T:Morozov} tells us that $\mathfrak g$ is of affine type or semi-simple in the two other cases.
    In the semi-simple case, $\mathfrak g$ is complete. Indeed, it is centerless, and all derivations are inner
    as shown in \cite[Chapter III, Theorem 9]{MR559927}.
    We can therefore conclude with Theorem \ref{THM:A} in the semi-simple case.

    On the other hand, when $\mathfrak g$ is a primitive Lie algebra of affine type, we conclude with Proposition \ref{P:affine} provided that $\mathfrak h$ is an \textit{algebraic} subalgebra of $\mathfrak{gl}(\mathfrak m)$ or, equivalently,  that one can integrate the isotropy Lie subalgebra  at $0$ given by $\varphi (\mathfrak h)$ (see proof of Proposition \ref{P:affine} ) to a \textit{closed} subgroup $H$ of $\mathrm{GL}(\C^n)\subset \mathrm{Aff}(\C^n)$. If this is the case, then
    we can identify $\mathbb C^n$, the target of the rational map given by Proposition \ref{P:affine}, with the homogeneous space $( H \ltimes \C^n)/H$ in such a way
    that $\varphi_* \mathfrak g$ is the Lie algebra of right invariant vector fields on $(H \ltimes \C^n)/H$.

    If $\mathfrak h$ is semi-simple, then the algebraicity of  $\mathfrak h$ follows from a result by Chevalley, which shows that
    the derived algebra $[\mathfrak g,\mathfrak g]$ of any lie algebra $\mathfrak g \subset \mathfrak gl(\C^n)$ is algebraic, see  \cite[Chapter II, Corollary 7.9]{MR1102012}. If instead $\mathfrak h$ is the direct sum of a semi-simple Lie algebra and a one-dimensional Lie algebra $\mathfrak L$, then the $\mathfrak h$-irreducibility of $\mathfrak m$ together with  Schur's Lemma imply that $\mathfrak L$ coincides with the center of $\mathfrak h$ and acts on $\mathfrak m$ by scalar multiplication, cf. discussion in \cite[Section 4]{MR2902724} after the statement of Morozov's Theorem. The algebraicity of $\mathfrak h$ follows.
\qed

\end{document}